\documentclass[a4paper]{amsart}

\usepackage{amssymb,amscd,stmaryrd}
\usepackage[mathcal]{eucal}
\usepackage{array,float}
\usepackage{enumitem}
\usepackage{mathtools}
\usepackage{xcolor}

\usepackage{xy}
\input xy
\xyoption{all}
\usepackage{pdflscape}
\usepackage{hyperref}

\usepackage{graphicx}

\usepackage[left=2cm, right=2cm]{geometry}

\numberwithin{equation}{section}
\numberwithin{equation}{subsection}

\newtheorem{thm}{Theorem}[section]

\newtheorem{corollary}[thm]{Corollary}
\newtheorem{lemma}[thm]{Lemma}

\newtheorem{definition}[thm]{Definition}
\newtheorem*{remark*}{Remark}

\newcommand{\ind}{{\mathrm{Ind}}}
\newcommand{\Hom}{{\mathrm{Hom}}}

\newcommand{\tra}{{\mathrm{tra}}}
\newcommand{\res}{{\mathrm{res}}}

\newcommand{\Ho}{{\mathrm{H}}}
\newcommand{\GL}{\mathrm{GL}}

\newcommand{\bZ}{{\mathbb Z}}
\newcommand{\bC}{{\mathbb C}}

\newcommand{\irr}{\mathrm{Irr}}
\newcommand{\wpsi}{\widetilde{\psi }}
\newcommand{\wg}{\widetilde{G}}
\newcommand{\wh}{\widetilde{H}}
\newcommand{\wrho}{\widetilde{\rho }}

\begin{document}
\title[On projective representations of finitely generated groups]{On  projective representations of finitely generated groups}
%\date{\today }
\author{Sumana Hatui}
\address{ SH: School of Mathematical Sciences, National Institute of Science Education and Research, An OCC of HBNI, Bhubaneswar 752050, Odisha, India.}
\email{sumanahatui@niser.ac.in}

\author{E. K. Narayanan}
\address{EKN: Department of Mathematics,
	Indian Institute of Science,
	Bangalore 560012, India. }
\email{naru@iisc.ac.in}

\author{Pooja Singla}
\address{ PS: Department of Mathematics and Statistics, Indian Institute of Technology Kanpur, Kanpur 208016, India. }
\email{psingla@iitk.ac.in}

\begin{abstract}
We prove a characterization of monomial projective representations of finitely generated nilpotent groups. We also characterize polycyclic groups whose projective representations are finite dimensional.
\end{abstract}

\maketitle
\section{Introduction}
The study of projective representations has a long history starting with the pioneering work of Schur for finite groups \cite{IS4,IS7, IS11}. It involves understanding homomorphisms from a group into the projective general linear groups.
Let $G$ be a group and $V$ be a complex vector space. A projective representation of a group $G$ is a homomorphism $\rho$  from  $G$ to  the projective general linear group $\mathrm{PGL}(V).$ More precisely, $\rho$ is a map from $G$ to the general linear group $\mathrm{GL}(V)$ with $\rho(1) = \mathrm{Id}_V$ and
such that $$ \rho(xy) = \alpha(x, y) \rho(x) \rho(y)~~\forall~x, y \in G$$ where $\alpha : G \times G \to \mathbb C^\times$ is a $2$-cocycle. Such a $\rho$ is called an $\alpha$-representation. If $\alpha(x, y) =1$ for all $x, y \in G$ then $\rho$ is an ordinary representation. A key concept in the study of projective representations is the {{\it representation group}} to which these representations lift as ordinary representations. Schur~\cite{IS4} proved that for every finite group $G$ there exists a group $\widetilde{G}$, nowadays called Schur cover or representation group of $G$, such that the $\alpha$-representations of $G$ are obtained from the ordinary representations of $\widetilde{G}.$ The representation groups for the symmetric group $S_n$ for $n \geq 4$, were classified by Schur \cite{IS11}. See \cite[Section 3.3]{GK} for more examples of Schur covers for several finite groups. For a finitely generated group $G$, under the assumption that $\Ho_2 (G, \mathbb Z)$ is finitely generated, there is a finitely generated representation group $\widetilde{G}$ (see \cite[Chapter II, Proposition 3.2]{Beyl}).

%In this paper we consider the following two questions related to the projective representations:\\

%\noindent
%(I) A characterization of monomial irreducible projective representations of a finitely generated nilpotent group. \\

%\noindent
%(II) A characterization of polycyclic groups whose projective representations are finite dimensional.\\

%In this article, we study projective representations that are monomial, that is, induced from a character of a subgroup (see Definition \ref{induced}).  We prove a characterization of irreducible monomial projective representations of finitely generated nilpotent groups.  Finally, in Section~\ref{Fdim} modifying the arguments in Hall~\cite[Theorem 3.3]{Ha}, we provide a characterization of polycyclic groups whose projective representations are finite dimensional and discuss some examples. 

%If $f : G \to \mathbb C^\times$ is any function, it is easy to see that the map $\alpha_f: G \times G \to \mathbb C^\times $ defined by $$\alpha_f(x, y) = f(xy)f(x)^{-1} f(y)^{-1}$$ is a $2$-cocycle. Such cocycles are called coboundaries and the set of coboundaries is denoted by $B^2 (G, \mathbb C^\times).$ Clearly, $B^2(G, \mathbb C^\times)$ is a subgroup of $Z^2(G, \mathbb C^\times).$

We next define induction for the projective representations of a group. A definition of this appeared in \cite[Section 2.2]{Cheng}.  Below, we give a slight variant of this one that fits better in our  discussion.
\begin{definition}[Induced projective representation] \label{induced}
	Let $\alpha$ be a $2$-cocycle of $G$, $H$ be a subgroup of $G$ and $(\phi, W)$ be an $\alpha$-representation of $H$.  The induced projective representation $(\widetilde{\phi}, V)$ of $\phi,$ denoted by $\ind_{H}^{G}(\phi),$ is defined as follows:
\noindent
The space $V$ consists of the functions $f:G \to W$ such that
	
	\begin{enumerate}[label=(\roman*)]
		\item  $f(hg)=\alpha(h, g )^{-1}\phi(h)f(g)  ~for all~ h \in H, ~ g \in G$.
		\item The support of $f$ is contained in a union of finitely many right cosets of $H$ in $G$.
	\end{enumerate}
The projective representation $\widetilde{\phi}: G \to GL(V)$ is defined by $\widetilde{\phi}(g)f(x)=\alpha(x,g)f(xg)$ for all $x, g \in G$.
	
\end{definition}
It is easy to see that   $\widetilde{\phi}(g)f \in V$ and $\widetilde{\phi}$ is an $\alpha$-representation of $G$. For $\alpha = 1$, the above definition coincides with usual induction for discrete groups, see \cite[Definition~2.1]{SN}.
\\

Our first result is  a characterization of  the irreducible monomial projective representations of finitely generated nilpotent groups. Let $\alpha \in Z^2 (G, \mathbb C^\times)$ and let $\rho$ be an $\alpha$-representation of $G.$ Then $\rho$ is said to be {\it monomial} if there exists a subgroup $H \subset G$ and an $\alpha$-representation $\psi : H \to \mathbb C^\times$ such that $\rho$ is equivalent to $\ind_H^G(\psi)$. %(see Definition \ref{induced}).
In the context of ordinary representations of finitely generated nilpotent groups, monomial representations are characterized by the finite weight property(see \cite{SN}, \cite{BG}). Motivated by this result, we define the following:

 \begin{definition}\label{fw-def}
An $\alpha$-representation $(\rho, V)$ of a group $G$ is said to have finite weight if there exists a subgroup $H \subset G$ and an $\alpha$-representation $\psi : H \to \mathbb C^\times$ such that the space $$V_H(\psi) =
 \{ v \in V:~\rho(h)v = \psi(h)v~\forall h \in H \}$$ is a non-trivial finite dimensional space.
 \end{definition}
The following result gives a complete characterization of monomial irreducible projective representations of a finitely generated nilpotent group.
 \begin{thm}\label{thm2}
 	An irreducible $\alpha$-representation $\rho$ of a finitely generated nilpotent group is monomial if and only if  it has finite weight.
 \end{thm}

For a proof of this result, see Section~\ref{representation}. To prove this result, we show that the characterization of monomial irreducible projective representations of $G$ can be obtained from the corresponding characterization for the ordinary representations of its representation group.

Our next result is a generalization of the well known result of Hall~\cite[Theorem 3.2, Theorem 3.3]{Ha} which states that a polycyclic group $G$ has all irreducible representations finite dimensional if and only if $G$ is abelian by finite.
We extend this result to projective representations.
For a $2$-cocycle $\alpha$ of $G$, we say $G$ is {\it $\alpha$-finite}  if every irreducible  $\alpha$-representation of $G$ is finite dimensional. We obtain the following characterization of $\alpha$-finite polycyclic groups.

\begin{thm}\label{thm3}
	Let $G$ be a polycyclic group and $\alpha$ be a $2$-cocycle of $G$. Then $G$ is $\alpha$-finite if and only if there is a normal abelian subgroup $N$ of $G$ such that $[\alpha|_{\scriptscriptstyle N \times N}]$ is of finite order and $G/N$ is a finite group.
\end{thm}
The proof of this result is included in Section  \ref{Fdim} and is built on generalizing the ideas of Hall for $\alpha = 1$ case.  We conclude by characterizing the finite dimensional irreducible $\alpha$-representations of discrete Heisenberg groups of rank one.

\section{Preliminaries}
In this section we recall some standard definitions and results regarding projective representations of discrete groups. We refer the reader to \cite{GK} for related results in the case of finite groups. Recall that the second cohomology group $\Ho^2 (G, \mathbb C^\times)$ is defined to be the abelian group $Z^2(G, \mathbb C^\times) / B^2 (G, \mathbb C^\times)$,  where
$Z^2 (G, \mathbb C^\times)$ is the set of all $2$-cocyles of $G$ which form an abelian group under the pointwise multiplication and   $B^2 (G, \mathbb C^\times)$is the collection of all 2-coboundaries on $G$.
For any $\alpha \in Z^2(G, \mathbb C^\times)$, its image in $\Ho^2(G, \mathbb C^\times)$ is denoted by $[\alpha ]$. Two $2$-cocycles $\alpha_1, \alpha_2 \in Z^2(G, \mathbb C^\times)$ are called cohomologous if $[\alpha_1] =  [\alpha_2]$. Let $V$ be a complex vector space. Recall, a projective representation of a group $G$ is a map $\rho: G \rightarrow \GL(V)$ such that
\[
\rho(x) \rho(y) = \alpha(x, y) \rho(xy), \,\, \text{for all}\,\,  x, y \in G,
\]
for suitable scalars $\alpha(x, y) \in \mathbb C^\times$. By the associativity of  $\GL(V)$, the map $(x,y) \mapsto \alpha(x, y)$ gives a 2-cocycle of $G$, that is, $\alpha$ satisfies the following:
$$ \alpha (x, y) \alpha(xy, z) = \alpha(x, yz) \alpha(y, z), \,\, \text{for all}\,\, x, y, z \in G.$$
In this case, we say $\rho$ is an $\alpha$-representation.  Two $\alpha$-representations $\rho_1: G \rightarrow \GL(V)$ and $\rho_2: G \rightarrow \GL(W)$ are called linearly equivalent if there is an invertible $T\in  \mathrm{Hom}( V, W)$ such that
\[
T \rho_1(g) T^{-1} = \rho_2(g), \,\, \text{for all}\,\,  g \in G.
\]
Recall that an $\alpha$-representation of $G$ for $\alpha(x, y) = 1$ for all $x,y \in G$ is called an ordinary representation of $G$. At times we shall  call this just as a representation of $G$, omitting the word ordinary, whenever our meaning is clear from the context.

Let $\irr(G)$ denote the set of all linearly inequivalent ordinary  irreducible representations of $G$ over $\bC$ and  $\irr^\alpha(G)$ denote the set of all linearly inequivalent irreducible $\alpha$-representations of $G$ over $\bC$. We remark that for $\alpha, \alpha' \in Z^2(G, \mathbb C^\times)$ such that $[\alpha] = [\alpha']$, the sets $\mathrm{Irr}^{\alpha'}(G) $ and $\mathrm{Irr}^{\alpha}(G) $  are in bijective correspondence and can be easily obtained from each other. Therefore to study irreducible projective representations of $G$, we will pick a representative $\alpha$ for each element of $\Ho^2(G, \mathbb C^\times)$ and study the corresponding $\alpha$-representations.

For a group $G$ and a $2$-cocycle $\alpha$ of $G$, the set $\mathbb C^\alpha G $, called the twisted group algebra of $G$ with $2$-cocycle $\alpha$, is a $\mathbb{C}$-algebra with its vector space basis given by the set $\{e_g \mid g \in G  \}$. The multiplication of basis elements of $\mathbb C^\alpha G $ is given by the following
\[
e_g e_h = \alpha(g, h ) e_{gh}, \,\, \text{for all}\,\, g,h \in G
\]
 and is extended linearly to the whole set.
Parallel to the ordinary representations of $G$, it is easy to see that $(\rho, V)$  is an $\alpha$-representation of $G$ if and only if $V$ is a $\mathbb C^\alpha G$-module where the action is via $\rho$.   The notion of twisted group algebra appears in \cite{GK} for finite groups and in \cite{passman1970radicals,passman1996semiprimitivity} for infinite groups.

We next recall the definition of transgression and inflation homomorphisms. For a central extension,
\[
1 \rightarrow A \rightarrow   \widetilde{G} \rightarrow \widetilde{G}/A \rightarrow 1,
\]
the Hochschild-Serre spectral sequence \cite{HS} for cohomology of groups yields the following exact sequence
\begin{equation}\label{centralequation}
\Hom(\widetilde{G}, \bC^\times) \xrightarrow{\res} \Hom( A,\bC^\times) \xrightarrow[]{\tra} \Ho^2(\widetilde{G}/A, \bC^\times)  \xrightarrow{\inf} \Ho^2(\widetilde{G}, \bC^\times),
\end{equation}
where  $\tra:\Hom( A,\bC^\times) \to \Ho^2(\widetilde{G}/A, \bC^\times) $ given by $f \mapsto \tra(f) = [\alpha]$, with
$$\alpha(\overline{x},\overline{y}) = f(\mu (\overline{x})\mu(\overline{y})\mu(\overline{xy})^{-1}), \,\, \text{for all}\,\,  \overline{x}, \overline{y} \in \widetilde{G}/A, $$
for a section $\mu: \widetilde{G}/A \rightarrow \widetilde{G}$, denotes the transgression homomorphism. The inflation homomorphism, $\inf : \Ho^2(\widetilde{G}/A, \bC^\times)   \to  \Ho^2(\widetilde{G}, \bC^\times) $ is given by $[\alpha] \mapsto \inf([\alpha]) = [\beta]$, where $\beta(x,y) = \alpha(xA,yA)$, for all $x,y \in \widetilde{G}$. \\
Throughout this article, while making a choice of a section map we will always choose one that maps identity to identity. We end this section with some required results regarding representation group of $G$.

\begin{definition}[Representation group of $G$]
	A group $\widetilde{G}$ is called a  \emph{representation group} of $G,$ if there is a central extension
	$$1 \rightarrow A \rightarrow   \widetilde{G} \rightarrow G \to 1$$ such that
	corresponding transgression map $$\tra: \Hom(A,\bC^\times) \to  \Ho^2(G,\bC^\times)$$ is an isomorphism.
\end{definition}

	\begin{lemma}\label{polycyclic}
		Let $G$ be a polycyclic group. Then $G$ has a representation group over $\bC$ which is finitely generated.  Furthermore if $G$ is finitely generated nilpotent, then $G$ has a representation group which is finitely generated nilpotent.
	\end{lemma}
	\begin{proof} For any group $G$, there exists a central extension $1 \to \Ho_2(G, \mathbb Z) \to \widetilde{G} \to G \to 1$ such that $\widetilde{G}$ is a representation group of $G$ (see \cite[Chapter II, Proposition 3.2]{Beyl}). In particular, if both $G$ and $\Ho_2(G, \bZ)$ are finitely generated then $\widetilde{G}$ is finitely generated. It is well known that  polycyclic groups are finitely presented. Let $G=F/R$ be a free presentation of polycyclic group $G$. By Hopf formula \cite{Hopf}, $\Ho_2(G, \bZ) \cong \frac{[F,F]\cap R}{[F,R]}$ is a subgroup of $R/[F,R]$ and $R/[F,R]$ is a finitely generated abelian group. Hence $\Ho_2(G, \bZ)$ is a finitely generated group. This proves the existence of a finitely generated representation group of a polycyclic group $G$. The result for nilpotent groups follows because every nilpotent group is a polycyclic group. 
	\end{proof}

In \cite{PS}, the authors described a representation group for finitely generated abelian groups as well as for discrete Heisenberg groups  and their $t$-variants. Then, by \cite[Corollary 3.3]{PS}, it follows that the sets $\mathrm{Irr}(\widetilde{G})$ and $\cup_{[\alpha]\in \Ho^2(G, \mathbb C^\times)} \mathrm{Irr}^\alpha(G)$ are in bijective correspondence.\\

\begin{lemma}
	\label{remark:injective-transgression}
Let $1 \rightarrow A \rightarrow   \widetilde{G} \rightarrow G \rightarrow 1$ be a central extension such that $\widetilde{G}$ is a representation group of $G$. Then
$A \subseteq [\widetilde{G}, \widetilde{G}]$, where $[\widetilde{G}, \widetilde{G}]$ denotes the commutator subgroup of $\widetilde{G}$.
\end{lemma}

\begin{proof}
By the definition of the representation group and the exactness of (\ref{centralequation}), we have $\text{res}: \Hom(\widetilde{G}, \bC^\times) \rightarrow \Hom( A,\bC^\times)$ is trivial. Hence, $A \subseteq [\widetilde{G}, \widetilde{G}]$.
	
	\end{proof}

\section{Monomial and finite weight projective representations}\label{representation}
In this section we explore the monomial projective representations of a finitely generated group $G$ and its relations to the monomial representations
of a representation group $\widetilde{G}.$ As a consequence we show that, for a finitely generated nilpotent group $G,$ an irreducible projective representation $\rho$ is monomial if and only if $\rho$ is of finite weight (see Definition \ref{fw-def}). We need the following results.

\begin{lemma}\label{diamond-lemma}

Let $H$ and $K$ be two subgroups of $G$ and $\chi: H \rightarrow
\mathbb C^{\times}$, $\delta: K \rightarrow \mathbb C^{\times}$ be
characters of $H$ and $K$ respectively such that,

\begin{enumerate}

\item $kHk^{-1} \subseteq H$ for all $k \in K$, i.e. $K$
normalizes $H$.

\item $\chi(khk^{-1}) = \chi(h)$ for all $h \in H$ and $k \in K$.

\item $\chi|_{H \cap K} = \delta|_{H \cap K}$

\end{enumerate}

Then $\chi \delta: HK \rightarrow \mathbb C^{\times}$ defined by
$\chi \delta(hk) = \chi(h) \delta(k)$ for all $h \in H$ and $k \in
K$ is a character of $HK$ such that $\chi \delta|_{H} = \chi$.

\end{lemma}

\begin{proof}
See Lemma 2.9 in \cite{SN}.

\end{proof}

\begin{thm}\label{frobenius}[Frobenius reciprocity]
 The induction functor is left adjoint to the restriction functor, i.e.,
\[
\mathrm{Hom}_G(\ind_H^G(\rho), \delta) = \mathrm{Hom}_K(\rho, \mathrm{Res}^G_K(\delta)),
\]
where $\delta$ is any representation of $G,$ $\rho$ is a representation of a subgroup $H$ of $G$ and $\mathrm{Res}^G_K(\delta)$ is the restriction of
$\delta$ to $K.$

\end{thm}

\begin{proof}
See \cite[Chapter~1]{V}
\end{proof}

\begin{thm}\label{monomial-ordinary-thm}
Let $G$ be a finitely generated nilpotent group. An irreducible representation $\rho$ of $G$ is monomial if and only if $\rho$ is of finite weight.
\end{thm}

\begin{proof}

See the main results in \cite{BG} and \cite{SN}.

\end{proof}

Let $G$ be a group and $\widetilde{G}$ be its representation group. Thus by Lemma~\ref{remark:injective-transgression}, we have a subgroup $A$ of $\widetilde{G},$
$A \subset [\widetilde{G}, \widetilde{G}] \cap Z(\widetilde{G})$ and a central extension $$1 \rightarrow A \rightarrow \widetilde{G} \xrightarrow{\pi} G \rightarrow 1$$ such that the map $\tra$ is an isomorphism. Fix a section $s: G \to \widetilde{G}$. Let $\alpha \in \Ho^2(G, \mathbb C^\times)$ and $\chi : A \to \mathbb C^\times$ be such that $$\alpha(x, y) = \chi(s(x)s(y) s(xy)^{-1}), \,\, \text{for all}\,\,  x, y \in G. $$
For any $\alpha$-representation $(\rho, V)$ of $G$,
define $\widetilde{\rho}: \widetilde{G} \rightarrow \text{GL}(V)$ by $\widetilde{\rho}(a_g s(g)) = \chi(a_g)\rho(g) $.  Then $(\widetilde{\rho}, V)$ is an ordinary representation of $\widetilde{G}$. It is well known that the map $\rho \mapsto \tilde{\rho}$ preserves irreducibility and gives an equivalence between categories of all $\alpha$-representations of $G$ and the category of all representations of $\widetilde{G}$ lying above $\chi$.

\begin{lemma}
\label{lem:suitable-H}
  Let $\widetilde{H}$ be a subgroup of $\widetilde{G}$ and $\widetilde{\psi}$ be a one dimensional ordinary representation of $\widetilde{H}$. If $\ind_{\widetilde{H}}^{\widetilde{G}} (\widetilde{\psi})$ is an ordinary irreducible representation of $\wg$, then $A \subseteq \widetilde{H}.$
\end{lemma}
\begin{proof} Suppose $A \nsubseteq \wh $. The character $\widetilde{\psi}$ restricted to $\widetilde{H} \cap A$ can be extended to a character $\delta$ of $A$ (since $A$ is abelian). By Lemma \ref{diamond-lemma} we obtain an extension of the character $\widetilde{\psi}$ to $\widetilde{\psi} \delta : \widetilde{H}A \to \mathbb C^\times.$ It follows from Frobenius reciprocity (Theorem \ref{frobenius}) that $$ \Hom_{\widetilde{H}A} \left ( \ind_{\widetilde{H}}^{\widetilde{H}A} (\widetilde{\psi}), \widetilde{\psi}\delta \right ) =
\Hom_{\widetilde{H}}\left ( \widetilde{\psi}, \widetilde{\psi }\delta|_{\widetilde{H}} \right ) \neq 0.$$ Hence $\ind_{\widetilde{H}}^{\widetilde{H}A} (\widetilde{\psi})$ is not irreducible. On the other hand $ \ind_{\widetilde{H}}^{\widetilde{G}} (\widetilde{\psi}) \cong \ind_{\widetilde{H}A}^{\widetilde{G}} \left ( \ind_{\widetilde{H}}^{\widetilde{H}A} (\widetilde{\psi}) \right )$ implies that $\ind_{\widetilde{H}}^{\widetilde{H}A} (\widetilde{\psi})$ is irreducible. This is a contradiction. Hence $A \subseteq \wh$. 
\end{proof}

\begin{proof}[\textbf{Proof of Theorem \ref{thm2}}]

	By Lemma~\ref{polycyclic}, a finitely generated nilpotent group $G$ has a representation group $\widetilde{G}$ which is also finitely generated nilpotent.
	Let $\rho : G \to GL(V)$ be an irreducible monomial $\alpha$-representation. Then there exists a subgroup $H \subset G$ and an $\alpha$-representation $\psi: H \to \mathbb C^\times$ such that $\rho = \ind_H^G(\psi).$  Let $\widetilde{H} = \pi^{-1}(H),$ where $\pi$ is the surjective homomorphism from $\widetilde{G}$ to $G = \widetilde{G}/ A$. Then $\widetilde{H}$ is a subgroup of $\widetilde{G}$ and every element $\widetilde{h} \in \widetilde{H}$ can be uniquely written as $\widetilde{h} = a_h s(h)$ for some $h \in H$ and $a_h \in A.$
Define $\widetilde{\psi}: \widetilde{H} \to \mathbb C^\times$ by $$\widetilde{\psi}(\widetilde{h}) = \widetilde{\psi} (a_h s(h)) = \chi(a_h) \psi(h).$$ The map $\widetilde{\psi}$ is a one dimensional representation of $\widetilde{H}.$ By definition of $\wrho$ and induced representation, we obtain that $(\widetilde{\rho}, V)$ and $\ind_{\widetilde{H}}^{\widetilde{G}}(\widetilde{\psi})$ are isomorphic as $\wg$ representations. By Frobenius reciprocity, 
\[
\Hom_{\wh}(\wpsi, \wrho|_{\wh}) =  \Hom_{\wg}(\wrho, \wrho). 
\]
Since $\wrho$ is irreducible, $V_{\wh}(\wpsi)$ is one dimensional. By definition of $\wrho$, $V_{\wh}(\wpsi) = V_{H}(\psi),$ hence $V_{H}(\psi)$ is finite dimensional. This implies that $\rho$ is a finite weight representation of $G$. Conversely, suppose $(\rho, V)$ is a finite weight irreducible $\alpha$-representation of $G$. Then there exists a subgroup $H$ and $\alpha$-representation $\psi: H \rightarrow \mathbb C^\times$ such that $V_{H}(\psi)$ is finite dimensional. Let $\widetilde{H} = \pi^{-1}(H)$ and $\widetilde{\psi}: \widetilde{H} \to \mathbb C^\times$ given by  $$ \widetilde{\psi}(\widetilde{h}) = \widetilde{\psi} (a_h s(h)) = \chi(a_h) \psi(h) $$
is a one dimensional character of $\wh$. We note that $(\wrho, V)$ is an irreducible representation of $\wh$ such that $V_{\wh}(\wpsi) = V_H(\psi).$ Therefore $\wrho$ is a finite weight representation. By Theorem~\ref{monomial-ordinary-thm}, $\wrho$ is a monomial representation. Let $\widetilde{\rho} \cong \ind_{\wh}^{\wg}(\wpsi)$. By Lemma~\ref{lem:suitable-H} $A \subseteq \wh$. Therefore $\rho \cong \ind_H^G(\psi)$ and hence a monomial representation.

\end{proof}

\noindent

\section{Criterion for finite dimensional projective representations}\label{Fdim}
Let $G$ be a polycyclic group. It is a well known result due to Hall that every irreducible representation of $G$ is finite dimensional if and only if $G$ is abelian by finite (see \cite[Theorem 3.2, Theorem 3.3]{Ha}). That is, such groups $G$ are characterized by the condition that there exists an abelian normal subgroup $N$ of $G$ such that $G/N$ is finite.
By closely following some of the arguments in \cite{Ha}, we extend this result to projective representations. For $\alpha \in Z^2 (G, \mathbb C^\times),$ we show that every irreducible $\alpha$-representation of $G$ is finite dimensional if and only if there exists an abelian normal subgroup $N$ of $G$ such that $G/N$ is finite and $[\alpha_{\scriptscriptstyle N \times N}]$ is of finite order.
Notice that the finite order condition is automatically satisfied if $[\alpha] = [1].$ As a consequence of this main result we show that {{\it every}} irreducible projective representation of a finitely generated polycyclic group $G$ is finite dimensional if and only if there is a normal abelian subgroup $N$ of $G$ such that $G/N$ is finite and $[\alpha|_{\scriptscriptstyle N \times N}]$ is of finite order for {{\it all}} $\alpha \in Z^2 (G, \mathbb C^\times).$

Recall that for $\alpha \in Z^2(G,\bC^\times)$, the group $G$ is $\alpha$-finite  if every irreducible $\alpha$-representation of $G$ is finite dimensional. We need the following results:

\begin{thm}\label{twostep-c}
	Let $N$ be a finitely generated abelian group. Then there exists a central extension $$ 1 \rightarrow Z \rightarrow N^\star \rightarrow N \rightarrow 1,$$ such that $Z = Z(N^\star) = [N^\star, N^\star]$ and $N^\star$ is a representation group of $N.$ Moreover, $N^\star$ is a finitely generated two step nilpotent group.
\end{thm}

\begin{proof}	See \cite[Theorem 1.4]{PS} and its proof. The proof of \cite[Theorem 1.4]{PS} is done by showing that the corresponding transgression map is an isomorphism.
\end{proof}
\begin{thm}\label{fd-two-step}
	Let $G$ be a finitely generated two step nilpotent group and $\pi$ an irreducible representation of $G.$ Then, $\pi$ is finite dimensional if and only if the character obtained by restricting $\pi$ to $[G, G]$ is of finite order.
	
\end{thm}

\begin{proof}
	See \cite[Theorem 1.3]{SN}.
\end{proof}

We start with the following lemma:

\begin{lemma}\label{abelian}
	Let $N$ be a finitely generated abelian group and $\alpha \in Z^2(N,\bC^\times)$.
	Then $N$ is $\alpha$-finite if and only if $[\alpha]$ is of finite order.
\end{lemma}

\begin{proof}
	Let $\rho$ be an irreducible $\alpha$-representation of $N$ and $\widetilde{\rho}$ be its lift to $N^\star$ (see Theorem \ref{twostep-c}).
	Let $s:N \to N^\star$ be a section such that $s(1)=1$.
	By Theorem \ref{twostep-c}, we have $\chi:Z \to \mathbb C^\times$ such that $\tra({\chi}) = [\alpha].$
	For $n_1, n_2 \in N,$  $$\alpha(n_1, n_2) = \rho(n_1) \rho(n_2) \rho (n_1n_2)^{-1} = \widetilde{\rho}(s(n_1)s(n_2)s(n_1n_2)^{-1}).$$ It follows from the injectivity of the map $\tra$ that
	the character obtained by restricting the irreducible representation $\widetilde{\rho}$ to the center $Z = [N^\star, N^\star]$ equals $\chi.$ Since orders of $\chi$ and $[\alpha]$ are equal and $\rho$ is finite dimensional if and only if $\widetilde{\rho}$ is finite dimensional, the proof now follows from Theorem \ref{fd-two-step}.
\end{proof}

%Here $N$ is  a finitely generated abelian group. By the proof of \cite[Theorem 1.3]{PS},  there is a central extension
%\[
%1 \to Z \to N^\star \to N \to 1
%\]
%such that  $Z =Z(N^\star)= [N^\star,N^\star]$ and  $N^\star$ is a representation group of $N$ which is two step nilpotent.
%Since $\tra:  \Hom(Z,\bC^\times ) \to \Ho^2(N,\bC^\times)$ is an isomorphism, so there is a $\chi \in  \Hom(Z,\bC^\times )$ such that  $\tra(\chi)=[\alpha]$.
%Thus every $\alpha$-representation of $N$ is finite dimensional if and only if ordinary representations of $N^\star$ lying above $\chi$ are finite dimensional.
%Now by \cite[Theorem 1.3]{SN},  it follows that every $\alpha$-representation of $N$ is finite dimensional if and only if  $\chi$ is of finite order.
%Since $\tra(\chi^n)=[\alpha]^n$ and $\tra$ is injective, so   $\chi$ is of finite order if and only if $[\alpha]$ is of finite order.
%This complete our proof.
%\end{proof}

%For a group $G$ and a $2$-cocycle $\alpha$ of $G$, the set $\mathbb C^\alpha G $, called twisted group algebra of $G$ with $2$-cocycle $\alpha$, is a $\mathbb{C}$-algebra with its vector space basis given by the set $\{e_g \mid g \in G  \}$. The multiplication of basis elements of $\mathbb C^\alpha G $ is given by the following and is extended linearly to the whole set. \[ e_g e_h = \alpha(g, h ) e_{gh}, \,\, \text{for all}\,\, g,h \in G.  \] Parallel to ordinary representations of $G$, it is easy to see that $(\rho, V)$  is an $\alpha$-representation of $G$ is equivalent to $V$ being a $\mathbb C^\alpha G$-module where action is via $\rho$.

Let $N$ be a subgroup of $G$ and $\alpha$ be a $2$-cocycle of $G$. The restriction of cocycle $\alpha$ to the set $N \times N$ gives a $2$-cocycle of $N$ and this is denoted by either $\alpha|_{\scriptscriptstyle N \times N}$ or by $\alpha$ itself.
\begin{lemma}\label{normalfdim}
	Let $G$ be $\alpha$-finite and $N$ be a normal subgroup of $G$. Then $N$ is $\alpha|_{\scriptscriptstyle N \times N}$-finite.
\end{lemma}
\begin{proof}
	%We closely follow the proof by Hall in \cite[Theorem 3.3]{Ha} which was for $\alpha=1$.
	%We show that if there is an infinite dimensional  irreducible $\alpha |_{N \times N}$-representation of $N$, then there is an infinite dimensional irreducible $\alpha$-representation of $G$.

	Let  $Q$ be a maximal ideal of the twisted group algebra $\bC^\alpha N$.  Then there exists a maximal ideal $P$ of $\bC^\alpha G$ containing $Q$.
	Consider a natural non-zero morphism $\lambda: \bC^\alpha N/Q\to \bC^\alpha G/P$ of $\bC^\alpha N$-modules that sends the class of 1 to the class of 1. By condition of the lemma, the vector space $\bC^\alpha G/P$ is finite-dimensional. Next, maximality of $Q$ implies that $\lambda$ is injective. This shows finite-dimensionality of $\bC^\alpha N/Q$.

\end{proof}

The following lemma will be crucially used in the proof of Theorem \ref{thm3}.
\begin{lemma}
	\label{lem:existence-of-A}
	Let $G$ be a finitely generated polycyclic group with cyclic series
	\[
	1=G_0 \unlhd G_1 \unlhd G_2  \unlhd \cdots \unlhd G_n=G.
	\]
	Assume that $G_{n-1}$ has an abelian subgroup of finite index and $G_n$ does not have an abelian subgroup of finite index. Then there exists a free abelian subgroup A of finite rank, normal in G, and of finite index in $G_{n-1}$, satisfying the following:
	\begin{enumerate}
		\item $\mathrm{rank}(A) \geq 2$.
		\item  There exists $z \in G_n$ such that $L = \langle A, z \rangle $ has finite index in $G$ and $L/A$ is infinite.
		\item There exists a subgroup $B$ of $A$ such that $A = B \oplus \langle t \rangle$ for some $t \in A$ and $N_{L}(B) >A$, where  $N_{L}(B)$ is the the normalizer of $B$ in $L$.
	\end{enumerate}
	
\end{lemma}
\begin{proof}
	A  proof of the above result is included in the proof of \cite[Theorem~3.3 (p.~615)]{Ha}.
\end{proof}

Now we are in a position to prove Theorem \ref{thm3}. We provide necessary and sufficient conditions for a polycyclic group to be $\alpha$-finite.

\begin{proof}[\textbf{Proof of  Theorem \ref{thm3}:}]
	We first suppose that $N$ is a normal abelian subgroup of $G$ of finite index such that every irreducible $\alpha$-representation of $N$ is finite dimensional.  Let $V$ be an irreducible $\bC^\alpha G$-module. We will prove that $V$ is finite dimensional. Suppose not.
	Since $V$ is also $\bC^\alpha N$-module and $G$ is finitely generated, there exists a maximal $\bC^\alpha N$-submodule of $V$, say $W$. Then $V/W$ is an irreducible $\bC^\alpha N$-module and hence finite dimensional.
	
	Let $\{x_1=1, x_2,\ldots, x_t\}$ be a set of left coset representatives of $N$ in $G$. Now $N(x_i W) \subseteq x_i(N W) \subseteq  x_iW$ and $dim(V/x_iW)=dim(V/W)$.
	Consider $W_0=\cap_{i=1}^t x_iW.$ Then $V/W_0$ is a finite dimensional  $\bC^\alpha N$-module. If $x_j x_i N = x_k N$ and $w \in W_0,$ $x_j(x_iw)=x_knw \in W_0$. Hence $W_0$ is a $\bC^\alpha G$-submodule of $V$. Since $V$ is irreducible, either $W_0=V$ or $W_0=0$. Either case leads to
	finite dimensionality of $V.$

	Conversely,  Suppose $G$ is $\alpha$-finite. We prove that there is an abelian normal subgroup $N$ of $G$ such that $G/N$ is finite. Since $G$ is a finitely generated polycyclic group, $G$ has a series
	\[
	1=G_0 \unlhd G_1 \unlhd G_2  \unlhd \cdots \unlhd G_n=G
	\]
	such that $G_i/G_{i-1}$ is cyclic. We use induction on $n$. The group $G$ is cyclic for $n=1$. So the result is true in this case. Now assume $n >1$. We consider the following cases separately:
	\begin{enumerate}
		\item[(a)] $G_{n-1}$ does not have an abelian subgroup of finite index.
		\item[(b)] $G_{n-1}$ has an abelian subgroup of finite index.
	\end{enumerate}
	
	\noindent {\bf Case~(a):} First suppose that $G_{n-1}$ has no abelian subgroup of finite index. Hence by induction hypothesis $G_{n-1}$ has an infinite dimensional irreducible $\alpha$-representation. So by Lemma \ref{normalfdim}, $G$ has an infinite dimensional irreducible $\alpha$-representation and we are done.\\

	\noindent {\bf Case~(b):}  In this case, we prove the result by contradiction. Assume that $G$ does not have an abelian subgroup of finite index.
	%The construction of the subgroup $L^{'}$ is exactly as in the proof of \cite[Theorem 3.3]{Ha}, and we reproduce it here for the sake of completeness. However, for the construction of the representation we use previous results and arguments from \cite{Ha}.
	Let $A$ be the normal, abelian subgroup of $G$ obtained by Lemma~\ref{lem:existence-of-A}. If $[\alpha|_{\scriptscriptstyle A \times A}]$ is not of finite order, then by Lemma \ref{abelian}, $A$ is not $\alpha$-finite. By Lemma \ref{normalfdim}, $G$ is not $\alpha$-finite, a contradiction. Hence $[\alpha|_{\scriptscriptstyle A \times A}]$ is of finite order. By Lemma \ref{abelian}, $A$ is $\alpha$-finite.
	By  Theorem \ref{twostep-c},  we have  a central extension
	$$1 \to Z \to A^\star \to A \to 1$$
	such that $A^\star$ is a a representation group  of $A$, which is a two-step nilpotent group such that $Z = [A^\star, A^\star]$. There is a character $\chi$  of $Z$ such that  $\tra(\chi)=[\alpha_{\scriptscriptstyle A \times A}]$. Since $A$ is $\alpha$-finite, irreducible ordinary representations of $A^\star$ lying above $\chi$ are finite dimensional and by \cite[Theorem 1.1]{SN}, they are monomial. Thus irreducible $\alpha|_{\scriptscriptstyle A \times A}$-representations of $A$ are finite dimensional and monomial.  So for $\rho \in \irr^\alpha(A)$, there exists a finite index subgroup $H$ of $A$ and a character $\psi: H \to \bC^\times$ such that
	$\alpha(h_1,h_2)=\psi(h_1)\psi(h_2)\psi(h_1 h_2)^{-1}$ and  $\rho = \ind_H^A(\psi).$ Define a map $\mu: G \to \bC^\times$ by
	\[
	\mu(g)=\begin{cases}
	\psi^{-1}(g) & \mathrm {for}\,\, g \in H,\\
	1 & \mathrm{for}\,\, g \notin H.
	\end{cases}
	\]
	Then take
	$\alpha'(g_1,g_2)=\alpha(g_1,g_2)\mu(g_1)\mu(g_2)\mu(g_1g_2)^{-1}$ for all $g_1, g_2 \in G$. The cocycles
	$\alpha'$ and $\alpha$ are cohomologous and $\alpha'|_{\scriptscriptstyle H \times H}=1$. Let $[A:H] = \ell$.
	Consider the subgroup $C$ of $A$ generated by $\ell$-th power of elements of $A.$ Clearly $C \subset H.$  By definition, $C$ is a characteristic subgroup of $A$ of finite index such that $\alpha'|_{\scriptscriptstyle C \times C}=1$.

	Let  $M=\langle C, z\rangle.$
	We show that  $M$ has an irreducible $\alpha'$-representation of infinite dimension.
	Since $M$ is a finitely generated polycyclic group, we have a  central extension
	\[
	1 \to J \to M^\star \xrightarrow{\pi} M \to 1
	\]   such that  $M^\star$ is a representation group of $M$, existence follows from Theorem \ref{polycyclic}. Then there is a character $\chi : J \to \bC^\times$  such that  $\tra(\chi)=[\alpha'|_{\scriptscriptstyle M \times M}]$.
	Let  $\widetilde{C}=\pi^{-1}(C)$. Consider the central extension $1 \to J \to \widetilde{C} \to C \to 1$.
	Every element  $\widetilde{c} \in \widetilde{C}$ can be written as $\widetilde{c}=s(c)j$, $c \in C, j \in J$, where $s$ is a section from $M$ to $M^\star$.
	
	Recall that $A =B \oplus \langle t \rangle$ such that $A/B$ is infinite cyclic. Since $A/C$ is of finite index, there is a smallest positive integer $k$ such that $t^k \in C$.
	Now let $\lambda \in \bC^\times$  which is not a root of unity. For each integer $h = 0,  \pm 1 , \pm 2,\ldots $, we define a function $\rho_h$ on $\widetilde{C}$ by the rule that,  for any $\widetilde{c} \in \widetilde{C}$,  $\rho_h(\widetilde{c})=\rho_h(s(c)j)=\chi(j)\lambda^{\beta}$, where  the integer $\beta=\beta_h(c)$ is defined by the condition that $z^hcz^{-h}=(t^k)^\beta (\,\,\mathrm{mod}\,\, B)$.
	Let $\widetilde{c}_i=j_i s(c_i), i=1,2$.
	Now we have $\widetilde{c}_1\widetilde{c}_2 = j_1j_2 s(c_1) s(c_2)s(c_1c_2)^{-1}s(c_1c_2).$ so,
	\begin{eqnarray*}
		\rho_h(\widetilde{c}_1\widetilde{c}_2)
		&=& \chi(j_1)\chi(j_2) \chi(s(c_1) s(c_2)s(c_1c_2)^{-1})\lambda^{\beta_1+\beta_2}\\
		&=& \chi(j_1)\chi(j_2)\alpha'(c_1, c_2)\lambda^{\beta_1+\beta_2}\\
		&=& \chi(j_1j_2)\lambda^{\beta_1+\beta_2}\\
		&=&   \rho_h(\widetilde{c}_1)\rho_h(\widetilde{c}_2).
	\end{eqnarray*}
	Hence $\rho_h$ are one dimensional ordinary representations of $\widetilde{C}$ such that $\rho_h|_J=\chi$.
	Observe that $\{\rho_h| h=0,\pm 1 , \pm 2, \cdots\}$ are not equivalent.

	Since $M/C \cong \langle z \rangle$, we have $M^\star/\widetilde{C} \cong \langle s(z) \rangle$.
	Now we define $V=\oplus_{-\infty}^{\infty} \bC[v_m]$, where $v_m$ is a generator of space $\rho_m$ for $m=0,  \pm 1 , \pm 2,\ldots $. Define
	\[
	s(z) v_m =v_{m+1},  \,\,  \widetilde{c}  v_m  =\rho_m(\widetilde{c})v_m.
	\]
	Then $V$ is an infinite dimensional ordinary representation of $M^\star$ lying above $\chi$. The representation $V$ is easily seen to be  irreducible.
	Thus $V$ is an infinite dimensional  irreducible $\alpha'$-representation of  $M$.  So  by Lemma \ref{normalfdim},  $G$ will have an infinite dimensional  irreducible $\alpha'$-representation. The cocycles $\alpha'$ and $\alpha$ are cohomologous, so $G$ will have an infinite dimensional irreducible $\alpha$-representation, which is a contradiction.
\end{proof}

\begin{corollary}\label{findim}
	Every irreducible projective representation of a  polycyclic group $G$ is finite dimensional if and only if there is a normal abelian subgroup $N$ of $G$ such that
	for any $\alpha\in Z^2(G, \bC^\times)$, $[\alpha|_{\scriptscriptstyle N \times N}]$ is of finite order and $G/N$ is a finite group.
\end{corollary}
\begin{proof}
	Let $G$ be a polycyclic group.
	If  there is an abelian normal subgroup $N$ such that for any $\alpha \in Z^2(G, \bC^\times)$, $[\alpha|_{\scriptscriptstyle N \times N}]$ is of finite order and $G/N$ is  finite, then by Theorem \ref{thm3} every irreducible projective representation of $G$ is finite dimensional.
	
	Conversely, suppose every irreducible projective representation of $G$ is finite dimensional.
	By Theorem \ref{thm3}, it follows that, there is  an abelian  normal subgroup $N$ such that $G/N$ is  finite.
	If  for  some $\alpha \in Z^2(G,\bC^\times)$,  $[\alpha|_{\scriptscriptstyle N \times N}]$ is not of finite order,  then by Lemma \ref{abelian}  and   Lemma \ref{normalfdim}, there exists an infinite dimensional irreducible $\alpha$-representation of $G$, which is a contradiction.
\end{proof}

\subsection{Examples}
In this section, we discuss some examples of $\alpha$-finite groups.\\

\noindent {\bf Example~1:}  Consider the group $G=(\bZ/n\bZ \times \bZ) \rtimes \bZ$, where the multiplication is defined by $$(m_1,n_1,p_1)(m_2,n_2,p_2)=(m_1+m_2+p_1n_2(mod~ n),n_1+n_2, p_1+p_2).$$
By \cite[Lemma 2.2(ii)]{PS}, it follows that every $2$-cocycle $\alpha \in Z^2(G,\bC^\times)$, upto cohomologous,  is of the following form:
\[
\sigma((m_1,n_1,p_1), (m_2,n_2,p_2)) =
\lambda^{(m_2p_1+  \frac{n_2p_1(p_1-1)}{2} )}\mu^{(n_1m_2+p_1\frac{n_2(n_2-1)}{2} + p_1n_1n_2)},
\]
where $\lambda, \mu \in \bC^\times$ such that $\lambda^n=\mu^n=1$.
Hence every $2$-cocycle is of finite order and there is a normal subgroup $(\bZ/n\bZ \times n\bZ) \times \bZ$ of $G$ such that quotient group is finite. So by Corollary \ref{findim}, every projective representation of $G$ is finite dimensional.\\

\noindent{\bf Example~2:}
Our next example is of  generalized discrete Heisenberg groups. These are finitely generated  two-step nilpotent groups of rank $2n + 1$ with rank $1$ center.
Given an $n$-tuple $(d_1, d_2, \ldots, d_n)$ of positive integers with $d_1|d_2|\cdots |d_n$ we write
$$G=\mathbb{H}_{2n+1}(d_1, d_2, \ldots, d_n) =\{(a, b,c) | a \in \bZ, b, c \in \bZ^n\},$$
where the group operation is defined  by
\begin{eqnarray*}
	(a,b_1,b_2,\ldots ,b_n,c_1,\ldots ,c_n)(a',b'_1,b'_2,\ldots ,b'_n,c'_1,\ldots ,c'_n)\\
	= (a+a'+\sum_{i=1}^{n} d_ib'_ic_i,b_1+b'_1,b_2+b'_2,\ldots ,b_n+b'_n,c_1+c'_1,\ldots ,c_n+c'_n)
\end{eqnarray*}
Consider $H=\mathbb{H}_3(d_1)$ and $K=\mathbb{H}_{2n-1}(d_2, \ldots, d_n)$.
Then $G$ is a central product of normal subgroups $H$ and $K$ with $Z=[H,H]\cap [K,K]=d_2 \bZ$. Consider the set $$ H_3^{d_1}(d_2)= \{(m,n,p) \mid m \in \bZ/d_2\bZ , n,p \in \bZ\},$$ with the group operation defined by
$$(m_1,n_1,p_1)(m_2,n_2,p_2)=(m_1+m_2+d_1p_1n_2,n_1+n_2, p_1+p_2).$$ Then
\begin{eqnarray*}
	G/Z  \cong H_3^{d_1}(d_2)  \times  \bZ^{2n-2}.
\end{eqnarray*}

By \cite[Corollary 3.2]{PS}, there is a bijective correspondence between the projective representations of $G$ and those of  $G/Z$.
Since $G/Z$ has a normal abelian subgroup $N= \bZ/d_2\bZ \times d_1\bZ \times \bZ^{2n-1}$ such that the quotient $\frac{G/Z}{N} \cong \bZ/d_1\bZ$. By Corollary \ref{findim},  every irreducible $\alpha$-representation of $G$ is finite dimensional if and only if $[\alpha_{\scriptscriptstyle N \times N}]$ is of finite order.\\

\noindent
{{\bf Acknowldgement:}}
%The authors thank the referee for asking pertinent questions and making suggestions which improved the exposition of this paper considerably.
The first named author acknowledges NBHM grant (0204/52/2019/R$\&$D-II/333) and IISc Raman post doctoral fellowship (R(IA)/CVR-PDF/2020/2700) for their support. The second and third named authors are grateful to the support in the form of SERB MATRICS grant MTR/2018/000501 and  MTR/2018/000094 respectively.

\bibliographystyle{amsplain}
\bibliography{projective}

\end{document}